\documentclass[11pt]{article}

\usepackage{amsmath, amssymb, amsthm}
\usepackage{mathtools}
\usepackage{enumitem}
\usepackage{hyperref}

\newtheorem{theorem}{Theorem}
\newtheorem{lemma}{Lemma}
\newtheorem{proposition}{Proposition}

\theoremstyle{definition}
\newtheorem{definition}{Definition}

\newtheorem{assumption}{Assumption}

\title{Hypothesis Testing over Observable Regimes in Singular Models}
\author{Sean Plummer\\ snplmmr@gmail.com}
\date{February 18, 2026}

\begin{document}
\maketitle

\begin{abstract}
Hypothesis testing in singular statistical models is often regarded as inherently problematic due to non-identifiability and degeneracy of the Fisher information. We show that the fundamental obstruction to testing in such models is not singularity itself, but the formulation of hypotheses on non-identifiable parameter quantities. Testing is inherently a problem in distribution space: if two hypotheses induce overlapping subsets of the model class, then no uniformly consistent test exists. We formalize this overlap obstruction and show that hypotheses depending on non-identifiable parameter functions necessarily fail in this sense. In contrast, hypotheses formulated over identifiable observables—quantities that are determined by the induced distribution—reduce entirely to classical testing theory. When the corresponding distributional regimes are separated in Hellinger distance, uniformly consistent tests exist and posterior contraction follows from standard testing-based arguments. Near singular boundaries, separation may collapse locally, leading to scale-dependent detectability governed jointly by sample size and distance to the singular stratum. We illustrate these phenomena in Gaussian mixture models and reduced-rank regression, exhibiting both untestable non-identifiable hypotheses and classically testable identifiable ones. The results provide a structural classification of which hypotheses in singular models are statistically meaningful.
\end{abstract}

\section{Introduction}
\label{sec:intro}

Hypothesis testing in parametric models is typically formulated as a question
about parameters: does $\theta$ belong to a null set $\Theta_0$ or an
alternative set $\Theta_1$?  In regular models, this formulation is
unproblematic.  Distinct parameter values induce distinct
data-generating distributions, the Fisher information is non-degenerate,
and classical asymptotic arguments provide conditions under which tests
with vanishing Type~I and Type~II errors exist (see, e.g., \cite{vdVaart1998}).

Many modern statistical models, however, are singular.  In singular models,
the parameterization may be non-identifiable, the Fisher information matrix
may be degenerate, and lower-dimensional strata may correspond to structurally
simpler representations.  Finite mixture models, reduced-rank regression,
neural networks, and latent variable models all exhibit such behavior (see, e.g., \cite{Watanabe2009,Drton2009}).
In these settings, distinct parameter values can induce the same
distribution.

This distinction is fundamental because testing is inherently a problem in
distribution space.  A test observes data only through its law.  If two
parameter values on opposite sides of a hypothesis generate the same
distribution, then no statistical procedure can distinguish them.
More generally, if a parameter-level hypothesis induces two subsets
$\mathcal{M}_0, \mathcal{M}_1 \subset \mathcal{M}$ of the model class
whose intersection is nonempty, then uniform separation is impossible.
The obstruction is therefore conceptual rather than technical:
the hypotheses do not correspond to distinct distributional regimes.

The central idea of this paper is to reformulate hypothesis testing in
singular models so that it respects identifiability.
Let $\Phi:\mathcal{W}\to\mathcal{M}$ denote a (possibly non-injective)
parameterization, $w \mapsto P_w$.
Two parameter values are observationally equivalent if they induce
the same distribution.
Statistical distinguishability is determined by equivalence classes
under this relation, not by raw parameter coordinates.

We therefore define hypotheses directly on the model class:
\[
H_0: P \in \mathcal{M}_0,
\qquad
H_1: P \in \mathcal{M}_1,
\]
where $\mathcal{M}_0,\mathcal{M}_1 \subset \mathcal{M}$ specify
distinct observable regimes.
A parameter-level hypothesis is statistically meaningful only if the
induced subsets of $\mathcal{M}$ are disjoint.

The results of this paper do not introduce new asymptotic rates or testing procedures. Rather, they provide a structural classification of hypothesis testability in singular models, clarifying which difficulties arise from non-identifiability and which do not. While the impossibility and separation results rely on classical testing theory, their synthesis yields a structural classification of hypothesis testability in singular models that does not appear explicitly in the literature. Existing discussions of singular models often attribute testing difficulties to geometric degeneracy. We show that the decisive criterion is instead whether the hypothesis respects observational equivalence. First, we formalize the overlap obstruction: if
$\mathcal{M}_0 \cap \mathcal{M}_1 \neq \varnothing$,
then no uniformly consistent test exists. Second, we show that hypotheses depending on non-identifiable
parameter quantities necessarily induce such overlap and are therefore
untestable, independently of estimation method or asymptotic regime. Third, we demonstrate that identifiable observables in singular models
fall entirely within classical testing theory: separation in Hellinger
distance yields uniformly consistent tests without modification of
asymptotic arguments. Fourth, we clarify how singular geometry affects finite-sample behavior.
Near singular boundaries, observable regimes may approach one another in
distribution space, leading to scale-dependent detectability consistent
with a local alternatives perspective. Finally, we illustrate these phenomena in Gaussian mixture models and
reduced-rank regression, exhibiting both testable identifiable hypotheses
and untestable non-identifiable ones.

The remainder of the paper develops these ideas formally.
Section~2 introduces observable regimes and identifiability.
Section~3 establishes the impossibility result for overlapping regimes.
Section~4 presents separation-based testability.
Section~5 discusses Bayesian contraction.
Section~6 examines scale dependence near singular boundaries,
and Section~7 provides numerical illustrations.

\section{Setup: Models, Observable Regimes, and Identifiability}
\label{sec:setup}

\subsection{Statistical model and parameterization}
Let $(\mathcal{X},\mathcal{F})$ be a measurable sample space and let
$\mathcal{M} \subset \mathcal{P}(\mathcal{X})$ be a statistical model.
We assume that $\mathcal{M}$ admits a (possibly non-injective)
parameterization
\[
\Phi : \mathcal{W} \to \mathcal{M}, \qquad w \mapsto P_w,
\]
with $\mathcal{W} \subset \mathbb{R}^d$.
The map $\Phi$ need not be injective; distinct parameter values may
induce the same distribution.

\subsection{Observational equivalence}
\label{sec:setup-equivalence}

Two parameter values $w, w' \in \mathcal{W}$ are said to be
\emph{observationally equivalent} if
\[
P_w = P_{w'}.
\]
This defines an equivalence relation on $\mathcal{W}$, which we denote by
$w \sim w'$.

The equivalence classes of $\sim$ consist of all parameter values that
induce the same distribution in $\mathcal{M}$.  Statistical distinguishability is therefore determined at the level of
$\mathcal{M}$ rather than at the level of individual parameter values.

\subsection{Observables}
\label{sec:setup-observables}

\begin{definition}[Observable]
An \emph{observable} is a measurable functional
$T : \mathcal{M} \to \mathbb{R}^k$.
\end{definition}

An observable assigns a numerical quantity to each distribution in the
model class. In particular, it is defined at the level of distributions
and is therefore invariant under changes of parameter representation
that leave the induced distribution unchanged.

A function $f : \mathcal{W} \to \mathbb{R}^k$ represents an observable
if it is constant on observational equivalence classes; that is,
\[
P_w = P_{w'} \Rightarrow f(w) = f(w').
\]

\begin{definition}[Identifiable observable]
A function $f : \mathcal{W} \to \mathbb{R}^k$ is \emph{identifiable}
if there exists a measurable functional
$T : \mathcal{M} \to \mathbb{R}^k$ such that $f(w) = T(P_w)$ for all $w \in \mathcal{W}$.
\end{definition}

An identifiable observable is therefore a parameter-level quantity
fully determined by the induced distribution. Hypotheses depending on
identifiable observables correspond to well-defined subsets of
$\mathcal{M}$, whereas hypotheses depending on non-identifiable
quantities may fail to do so.

\subsection{Hypotheses as observable regimes}
\label{sec:setup-regimes}

We formulate hypothesis testing problems directly on the model class $\mathcal{M}$. A pair of hypotheses is specified by two subsets
\[
H_0 : P \in \mathcal{M}_0,
\qquad
H_1 : P \in \mathcal{M}_1,
\]
where $\mathcal{M}_0, \mathcal{M}_1 \subset \mathcal{M}$. When $T : \mathcal{M} \to \mathbb{R}^k$ is an observable and
$A,B \subset \mathbb{R}^k$ are disjoint sets, we may define
observable regimes by
\[
\mathcal{M}_0 = \{ P \in \mathcal{M} : T(P) \in A \},
\qquad
\mathcal{M}_1 = \{ P \in \mathcal{M} : T(P) \in B \}.
\]
If $f : \mathcal{W} \to \mathbb{R}^k$ is an identifiable observable,
then the induced parameter subsets
\[
\mathcal{W}_0 = \{ w : f(w) \in A \},
\qquad
\mathcal{W}_1 = \{ w : f(w) \in B \}
\]
generate well-defined regimes in $\mathcal{M}$ via $\Phi$.
In contrast, if $f$ is not identifiable, the induced subsets of
$\mathcal{M}$ may overlap.

\section{Impossibility: Overlap and Non-identifiable Hypotheses}
\label{sec:impossibility}

\subsection{Tests and error criteria}
\label{sec:impossibility-tests}

Let $X_1,\dots,X_n$ be independent observations drawn from $P \in \mathcal{M}$. We write $P^n$ for the $n$-fold product measure corresponding to
independent observations from $P$. A (nonrandomized) test is a measurable function $\phi_n : \mathcal{X}^n \to \{0,1\}$. For hypotheses
\[
H_0 : P \in \mathcal{M}_0,
\qquad
H_1 : P \in \mathcal{M}_1,
\]
define the worst-case Type~I and Type~II errors by
\[
\alpha_n(\phi_n)
=
\sup_{P \in \mathcal{M}_0}
P^n(\phi_n = 1),
\qquad
\beta_n(\phi_n)
=
\sup_{Q \in \mathcal{M}_1}
Q^n(\phi_n = 0).
\]

\subsection{Overlap implies impossibility}
\label{sec:overlap}

\begin{lemma}[Overlap implies impossibility of uniform testing]
\label{lem:overlap}
If $\mathcal{M}_0 \cap \mathcal{M}_1 \neq \varnothing$, then for any test
$\phi_n$,
\[
\alpha_n(\phi_n) + \beta_n(\phi_n) \ge 1.
\]
Consequently, no sequence of tests can satisfy
$\alpha_n(\phi_n) \to 0$ and $\beta_n(\phi_n) \to 0$.
\end{lemma}

\begin{proof}
Let $P^\star \in \mathcal{M}_0 \cap \mathcal{M}_1$.
Then
\[
\alpha_n(\phi_n)
\ge
P^{\star n}(\phi_n = 1),
\qquad
\beta_n(\phi_n)
\ge
P^{\star n}(\phi_n = 0).
\]
Since
\[
P^{\star n}(\phi_n = 1)
+
P^{\star n}(\phi_n = 0)
=
1,
\]
the result follows.
\end{proof}

\subsection{Non-identifiable parameter hypotheses induce overlap}
\label{sec:nonidentifiable}

\begin{definition}[Non-identifiable quantity]
A function $f : \mathcal{W} \to \mathbb{R}$ is non-identifiable if
there exist $w, w' \in \mathcal{W}$ such that
\[
P_w = P_{w'} 
\quad \text{but} \quad
f(w) \neq f(w').
\]
\end{definition}

\begin{proposition}[Non-identifiable hypotheses are untestable]
\label{prop:nonid}
Let $f : \mathcal{W} \to \mathbb{R}^k$ be non-identifiable.
Suppose there exist disjoint sets $A,B \subset \mathbb{R}^k$ and
parameters $w_0,w_1 \in \mathcal{W}$ such that
\[
f(w_0) \in A, \qquad
f(w_1) \in B, \qquad
P_{w_0} = P_{w_1}.
\]
Define
\[
\mathcal{W}_0 = \{ w : f(w) \in A \},
\qquad
\mathcal{W}_1 = \{ w : f(w) \in B \},
\]
and let
\[
\mathcal{M}_i = \Phi(\mathcal{W}_i), \quad i=0,1.
\]
Then $\mathcal{M}_0 \cap \mathcal{M}_1 \neq \varnothing$.
Consequently, by Lemma~\ref{lem:overlap}, no uniformly consistent test
exists for distinguishing these hypotheses.
\end{proposition}

\begin{proof}
By assumption, $w_0 \in \mathcal{W}_0$ and
$w_1 \in \mathcal{W}_1$, while
$P_{w_0} = P_{w_1}$.
Hence $P_{w_0} \in \mathcal{M}_0 \cap \mathcal{M}_1$.
The conclusion follows from Lemma~\ref{lem:overlap}.
\end{proof}

This establishes that hypotheses depending on non-identifiable
parameter quantities are untestable.
We now turn to conditions under which observable regimes can be
separated in distribution space.

\section{Testability: Separation in Distribution Space}

\subsection{Hellinger distance and separated alternatives}
\label{sec:hellinger}

Assume that $\mathcal{M}$ is dominated by a common
$\sigma$-finite measure $\mu$.
For $P,Q \in \mathcal{M}$ with densities $p,q$
with respect to $\mu$, define the squared Hellinger distance
\[
h^2(P,Q)
=
\frac{1}{2}
\int
\left( \sqrt{p} - \sqrt{q} \right)^2
\, d\mu.
\]
We write $h(P,Q) = \sqrt{h^2(P,Q)}$.

\begin{definition}[$\varepsilon$-separated alternatives]
For $\varepsilon > 0$, define
\[
\mathcal{M}_1(\varepsilon)
=
\left\{
Q \in \mathcal{M}_1
:
\inf_{P \in \mathcal{M}_0}
h(P,Q)
\ge \varepsilon
\right\}.
\]
\end{definition}

\subsection{Existence of tests under separation}
\label{sec:separation-tests}

\begin{lemma}[Separated regimes admit tests]
\label{lem:separated}
If there exists $\varepsilon > 0$ such that
\[
\inf_{P \in \mathcal{M}_0,\; Q \in \mathcal{M}_1}
h(P,Q)
\ge \varepsilon,
\]
then there exists a sequence of tests $\phi_n$ such that
\[
\alpha_n(\phi_n)
\le
e^{-c n \varepsilon^2},
\qquad
\beta_n(\phi_n)
\le
e^{-c n \varepsilon^2}
\]
for some constant $c>0$.
\end{lemma}

\begin{proof}[Proof sketch]
This follows from classical results on testing separated hypotheses
under Hellinger distance (see, e.g., \cite{LeCam1986,vdVaart1998}).
.
\end{proof}

Lemma~\ref{lem:separated} shows that separation in distribution space
is sufficient for the existence of uniformly consistent tests.
Combined with Lemma~\ref{lem:overlap}, this establishes a sharp
dichotomy: overlap implies impossibility, while positive separation
implies testability.

The following proposition makes explicit that identifiable observables
in singular models fall within the scope of classical testing theory.

\begin{proposition}[Identifiable observable hypotheses reduce to classical testing theory]
\label{prop:identifiable-testable}
Let $f : \mathcal{W} \to \mathbb{R}^k$ be an identifiable observable,
with associated functional $T : \mathcal{M} \to \mathbb{R}^k$.
Suppose $A,B \subset \mathbb{R}^k$ are disjoint and satisfy
\[
\inf_{P \in \mathcal{M}_0,\; Q \in \mathcal{M}_1}
h(P,Q)
> 0,
\]
where
\[
\mathcal{M}_0 = \{ P : T(P) \in A \},
\qquad
\mathcal{M}_1 = \{ P : T(P) \in B \}.
\]
Then the hypotheses
\[
H_0 : f(w) \in A,
\qquad
H_1 : f(w) \in B
\]
admit uniformly consistent tests with exponentially small error.
\end{proposition}

\begin{proof}
This follows immediately from Lemma~\ref{lem:separated},
since identifiable observables induce well-defined
disjoint regimes in $\mathcal{M}$.
\end{proof}

Thus singularity does not preclude classical hypothesis testing.
What matters is whether the quantity of interest is identifiable
at the level of induced distributions.

\section{Bayesian Consequences: Contraction over Observable Regimes}

\subsection{Testing condition}
\label{sec:testing-condition}

Let $P_0 \in \mathcal{M}_0$ denote the true data-generating distribution. We study posterior concentration away from the alternative regime. We recall a standard testing-based condition for posterior contraction (see, e.g., \cite{Ghosal2000,vdVaart1998}).

\begin{assumption}[Existence of separating tests at scale $\varepsilon_n$]
\label{ass:testing}
There exists a sequence $\varepsilon_n \downarrow 0$ and tests
$\phi_n$ such that
\[
P_0^n(\phi_n = 1)
\le
e^{-c n \varepsilon_n^2},
\qquad
\sup_{Q \in \mathcal{M}_1(\varepsilon_n)}
Q^n(\phi_n = 0)
\le
e^{-c n \varepsilon_n^2}
\]
for some constant $c > 0$.
\end{assumption}
\subsection{Posterior contraction over observable regimes}
\label{sec:contraction}
The following result is a direct specialization of standard posterior contraction theory to observable regimes.
\begin{theorem}[Contraction over observable regimes]
\label{thm:contraction}
Suppose Assumption~\ref{ass:testing} holds and the prior $\Pi$
assigns sufficient mass to a Hellinger neighborhood of $P_0$, that is,
\[
\Pi\bigl(
P \in \mathcal{M}
:
h(P,P_0) \le \varepsilon_n
\bigr)
\ge
e^{-C n \varepsilon_n^2}
\]
for some constant $C>0$ and all sufficiently large $n$.
Then
\[
\Pi\bigl(
P \in \mathcal{M}_1(\varepsilon_n)
\mid X_1,\dots,X_n
\bigr)
\to 0
\]
in $P_0^n$-probability.
\end{theorem}

\begin{proof}[Proof sketch]
This is a direct specialization of standard posterior contraction
theorems based on testing conditions (see, e.g., \cite{Ghosal2000}). The only modification is that the hypotheses are defined
over subsets of $\mathcal{M}$ rather than over parameter space.
\end{proof}

\subsection{Failure of contraction under non-separation}
\label{sec:nonseparation}

\begin{proposition}[Non-separation obstructs testing and contraction]
\label{prop:nonseparation}
If
\[
\inf_{P \in \mathcal{M}_0,\; Q \in \mathcal{M}_1}
h(P,Q)
=
0,
\]
then no sequence of tests can satisfy
\[
\sup_{P \in \mathcal{M}_0} P^n(\phi_n=1) \to 0
\quad \text{and} \quad
\sup_{Q \in \mathcal{M}_1} Q^n(\phi_n=0) \to 0.
\]

Consequently, posterior contraction over $\mathcal{M}_1$
relative to $\mathcal{M}_0$ cannot occur uniformly at any fixed
positive separation scale.
\end{proposition}

\begin{proof}
If the infimum of Hellinger distance is zero, then for every
$\varepsilon>0$ there exist $P \in \mathcal{M}_0$ and
$Q \in \mathcal{M}_1$ with $h(P,Q) < \varepsilon$.
By classical lower bounds for testing under Hellinger distance,
uniformly consistent tests do not exist in this case.
The contraction statement follows from Theorem~\ref{thm:contraction}.
\end{proof}

This motivates the scale-dependent analysis developed in Section~6.

\section{Scale Dependence Near Singular Boundaries}

\subsection{Local alternatives viewpoint}
\label{sec:local-alternatives}

In singular models, the null and alternative regimes may intersect
or approach one another along lower-dimensional strata of the model.
In such cases,
\[
\inf_{P \in \mathcal{M}_0,\; Q \in \mathcal{M}_1}
h(P,Q)
=
0,
\]
even though the regimes are distinct away from the singular boundary. This situation is analogous to classical local alternatives.
Although fixed positive separation may fail,
distributions at distance $\delta$ from the boundary may satisfy
\[
h(P,Q) \asymp \delta^a
\]
for some exponent $a > 0$ determined by the local geometry.
Detectability then depends jointly on the sample size $n$
and the distance $\delta$ to the singular stratum. In particular, if $n \delta^{2a} \to \infty$,
consistent testing may be possible,
whereas if $n \delta^{2a}$ remains bounded,
the null and alternative are effectively indistinguishable.

\subsection{Generic scale–detectability behavior}
\label{sec:scale-generic}

The preceding discussion can be expressed in a generic form under
mild regularity assumptions. Suppose, heuristically, that near a singular boundary
the Hellinger distance between regimes satisfies
\[
h(P_\theta, P_0)
\asymp
\mathrm{dist}(\theta, \mathcal{S})^{\,a},
\]
where $\mathcal{S}$ denotes a singular stratum and $a>0$
is determined by the local structure of the model. Consider a sequence of alternatives $\theta_n$
approaching the boundary at rate $\delta_n$.
Then
\[
h(P_{\theta_n}, P_0)
\asymp
\delta_n^{\,a}.
\]
Classical testing lower bounds imply that detectability depends on the
asymptotic behavior of
\[
n \, \delta_n^{2a}.
\]
If $n \delta_n^{2a} \to \infty$,
consistent testing is possible along the sequence.
If $n \delta_n^{2a}$ remains bounded,
the null and alternative are asymptotically indistinguishable. This illustrates that in singular models,
the feasibility of testing is inherently scale-dependent near
singular strata.

The exponent $a$ reflects the order at which distributional separation emerges when approaching a singular boundary. In many singular models, this order is governed by the algebraic or geometric structure of the parameterization. Determining $a$ in specific classes of models is a
nontrivial problem that requires detailed local analysis and may involve tools beyond classical regular asymptotics. We do not attempt such a derivation here; instead, we use the generic scaling form to illustrate how singular geometry influences detectability.

\section{Numerical Illustrations}
\label{sec:experiments}

\subsection{Experiment grid: identifiable vs non-identifiable}
\label{sec:experiment-grid}

The experiments are designed to illustrate the structural dichotomy
established in Sections 3 and 4.  For each of two singular models—
Gaussian mixture models (GMM) and reduced-rank regression (RRR)—
we consider both identifiable and non-identifiable hypotheses.

This yields four cases:

\begin{center}
\begin{tabular}{lcc}
\hline
 & Identifiable hypothesis & Non-identifiable hypothesis \\
\hline
Gaussian mixture model & Testable & Untestable \\
Reduced-rank regression & Testable & Untestable \\
\hline
\end{tabular}
\end{center}

The goal of these experiments is not to study power in detail,
but to demonstrate the qualitative phenomena predicted by the theory:
overlap leads to persistent indistinguishability, while separation
of identifiable observable regimes yields classical testing behavior.

\subsection{Gaussian mixture model: non-identifiable hypothesis}
\label{sec:gmm-nonid}

We consider a two-component Gaussian mixture model with equal variances.
The parameterization exhibits label symmetry: permuting component labels
leaves the induced distribution unchanged. Label symmetry in mixture models is classical
(see, e.g., \cite{Teicher1963}).

We test the parameter-level hypothesis
\[
H_0 : \mu_1 > \mu_2,
\qquad
H_1 : \mu_1 < \mu_2,
\]
where $(\mu_1,\mu_2)$ denote the component means.
This hypothesis depends on the ordering of mixture components and is
therefore non-identifiable under label permutations. Indeed, for any parameter value $(\mu_1,\mu_2)$,
the relabeled parameter $(\mu_2,\mu_1)$ induces the same distribution.
Consequently, the induced subsets of $\mathcal{M}$ overlap,
and Proposition~1 implies that uniformly consistent testing is impossible. We estimate empirical Type~I and Type~II errors across increasing sample
sizes. Figure~\ref{fig:gmm-nonid} reports the resulting error curves.
In all cases, the sum
\[
\widehat{\alpha}_n + \widehat{\beta}_n
\]
remains close to one, indicating persistent indistinguishability.
Increasing the sample size does not resolve the ambiguity,
consistent with Proposition~\ref{prop:nonid} and
Lemma~\ref{lem:overlap}. The empirical behavior mirrors the theoretical lower bound:
increasing the sample size does not reduce the combined error below one,
since the hypotheses induce overlapping distribution classes.

\begin{figure}[t]
\centering
\includegraphics[width=0.75\textwidth]{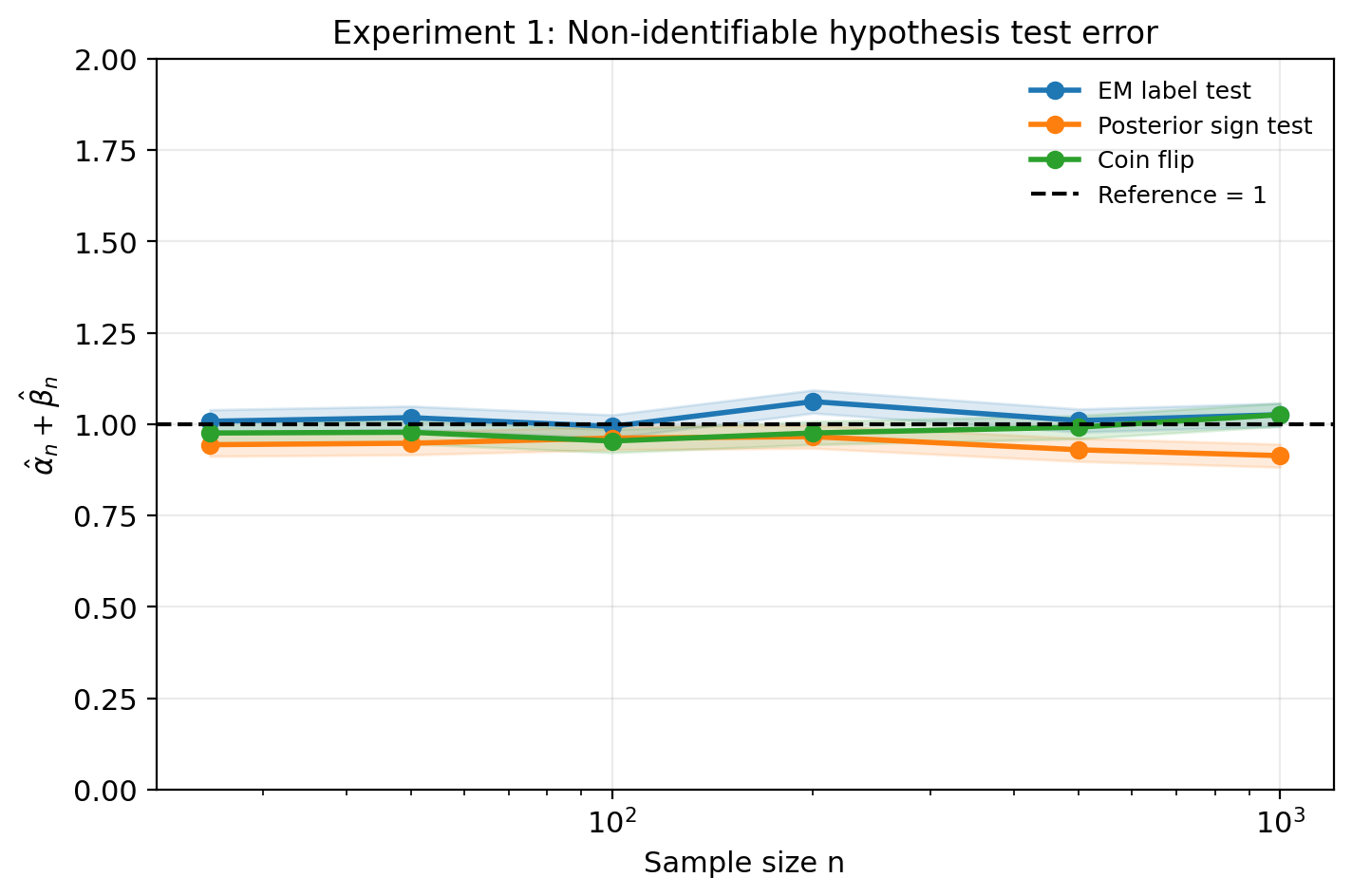}
\caption{
Empirical Type~I error $\widehat{\alpha}_n$,
Type~II error $\widehat{\beta}_n$,
and their sum for the non-identifiable
component-ordering hypothesis in a two-component
Gaussian mixture model.
Across increasing sample sizes,
$\widehat{\alpha}_n + \widehat{\beta}_n$
remains close to one, consistent with
Proposition~\ref{prop:nonid}. 
}
\label{fig:gmm-nonid}
\end{figure}

\subsection{Reduced-rank regression: identifiable rank hypothesis}
\label{sec:rrr-id}

We consider reduced-rank regression (RRR) with response dimension $q$
and predictor dimension $p$.
The model parameter is a matrix $B \in \mathbb{R}^{p \times q}$,
and the induced distribution depends only on the linear map
$x \mapsto Bx$.

We test the hypothesis
\[
H_0 : \operatorname{rank}(B) = r_0,
\qquad
H_1 : \operatorname{rank}(B) \ge r_0 + 1.
\]
The matrix rank is identifiable at the level of the induced
distribution: distinct ranks correspond to distinct covariance
structures in the regression model.
Thus the hypothesis is formulated in terms of an identifiable
observable. When the smallest nonzero singular value under the alternative
is bounded away from zero, the null and alternative regimes are
separated in Hellinger distance.
Proposition~\ref{prop:identifiable-testable} therefore implies that
uniformly consistent tests exist. Figure~\ref{fig:rrr-id} reports empirical Type~I error and power
across increasing sample sizes.
The size remains controlled near the nominal level,
while power increases with $n$,
illustrating classical separation behavior.

\begin{figure}[t]
\centering
\includegraphics[width=0.75\textwidth]{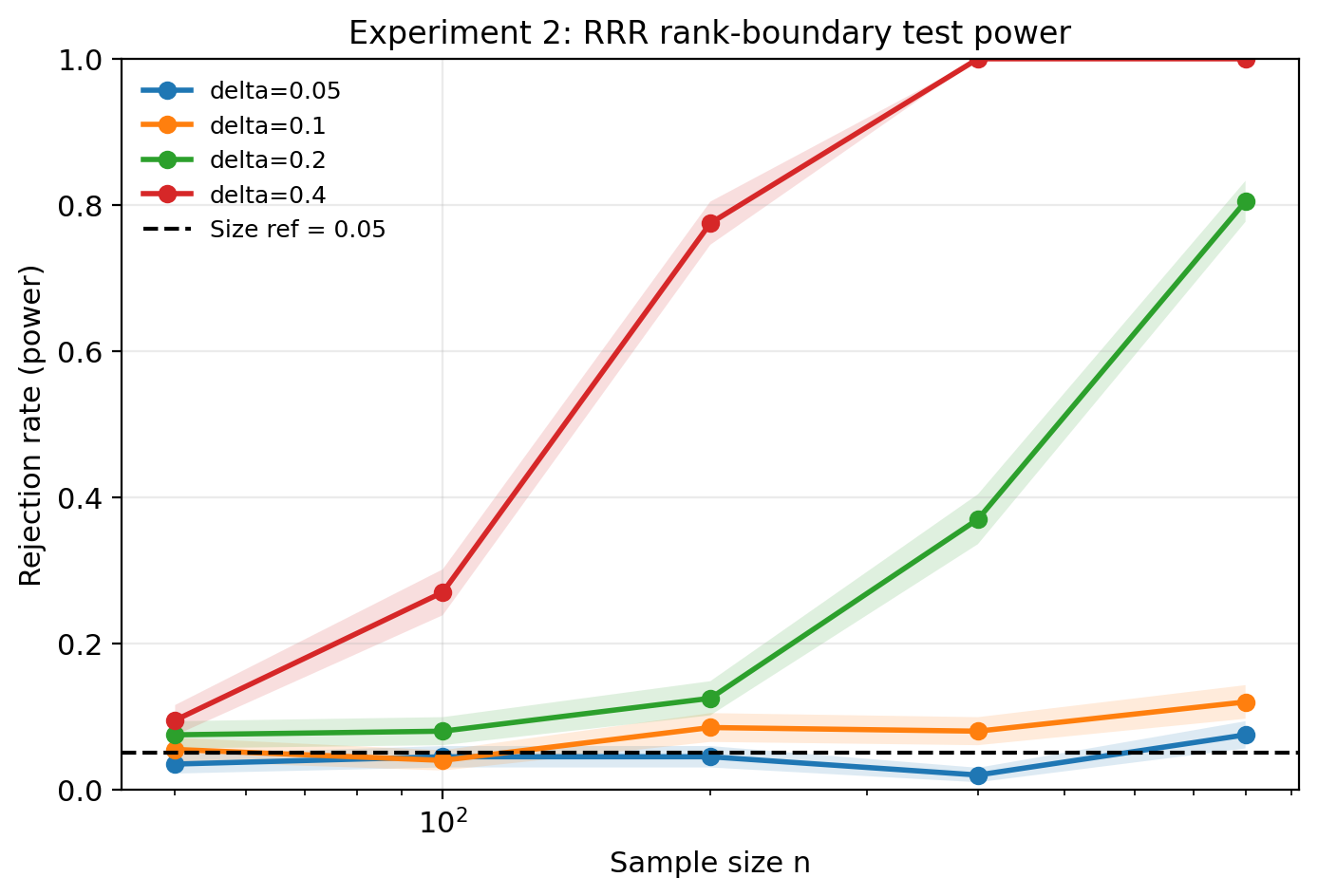}
\caption{
Empirical Type~I error and power for the identifiable
rank hypothesis in reduced-rank regression.
The smallest singular value under the alternative is
bounded away from zero.
Power increases with sample size,
consistent with classical separation theory.
}
\label{fig:rrr-id}
\end{figure}
 
\subsection{Gaussian mixture model: identifiable regime hypothesis}
\label{sec:gmm-id}

We return to the two-component Gaussian mixture model,
but now consider an identifiable observable regime.
Specifically, we test
\[
H_0 : P \text{ is a single Gaussian},
\qquad
H_1 : P \text{ is a two-component mixture with separated means}.
\]

The null corresponds to the regime in which the mixture
components coincide, while the alternative imposes a minimum
separation between component means.
These regimes are defined at the level of induced distributions,
and do not depend on component labeling. When the separation between component means is bounded away from zero,
the null and alternative are separated in Hellinger distance. Thus Proposition~\ref{prop:identifiable-testable} implies that
uniformly consistent tests exist. Figure~\ref{fig:gmm-id} reports empirical size and power as the
sample size increases.
The test maintains size under the null,
and power increases with $n$ when the mixture components are
sufficiently separated. This confirms that singularity alone does not prevent
classical hypothesis testing.

\begin{figure}[t]
\centering
\includegraphics[width=0.75\textwidth]{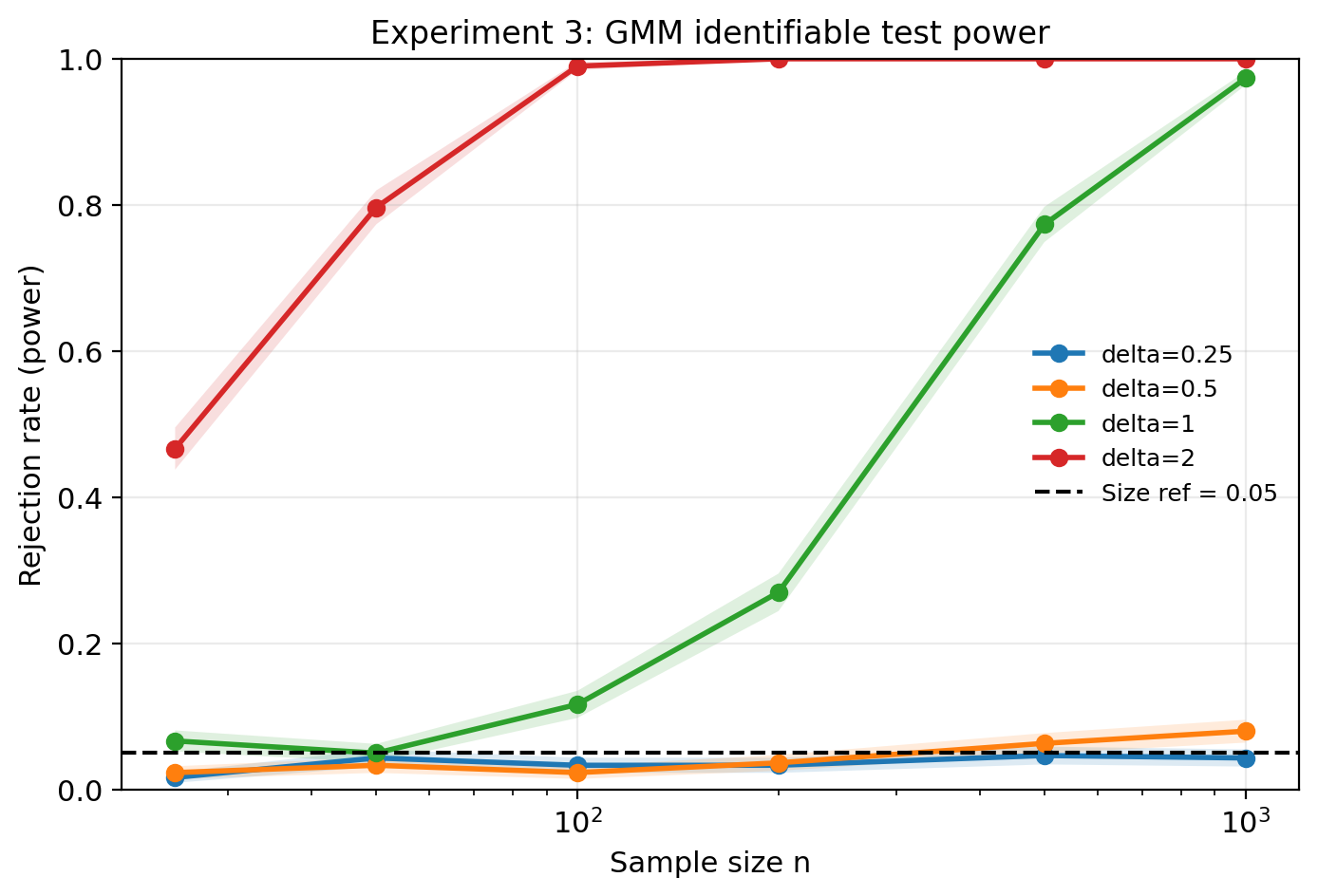}
\caption{
Empirical Type~I error and power for the identifiable
single-Gaussian versus separated-mixture hypothesis.
When component separation is bounded away from zero,
power increases with sample size,
illustrating classical testability in a singular model.
}
\label{fig:gmm-id}
\end{figure}

\subsection{Reduced-rank regression: non-identifiable hypothesis}
\label{sec:rrr-nonid}

We again consider reduced-rank regression,
but now test a parameter-level hypothesis that is
not identifiable at the level of induced distributions.

Let $B = U \Sigma V^\top$ denote a singular value decomposition.
We test
\[
H_0 : U_{1,1} > 0,
\qquad
H_1 : U_{1,1} < 0,
\]
where $U_{1,1}$ denotes the first entry of the left singular vector. This hypothesis depends on the sign of a singular vector,
which is not uniquely determined:
if $U$ is a singular vector matrix,
then so is $-U$,
and both induce the same regression distribution.
Thus the hypothesis depends on a non-identifiable quantity. The induced subsets of $\mathcal{M}$ therefore overlap,
and Proposition~\ref{prop:nonid} implies that uniformly
consistent testing is impossible. Figure~\ref{fig:rrr-nonid} reports empirical Type~I and Type~II
errors across increasing sample sizes.
As in the GMM non-identifiable case,
the sum
\[
\widehat{\alpha}_n + \widehat{\beta}_n
\]
remains close to one,
demonstrating persistent indistinguishability.

\begin{figure}[t]
\centering
\includegraphics[width=0.75\textwidth]{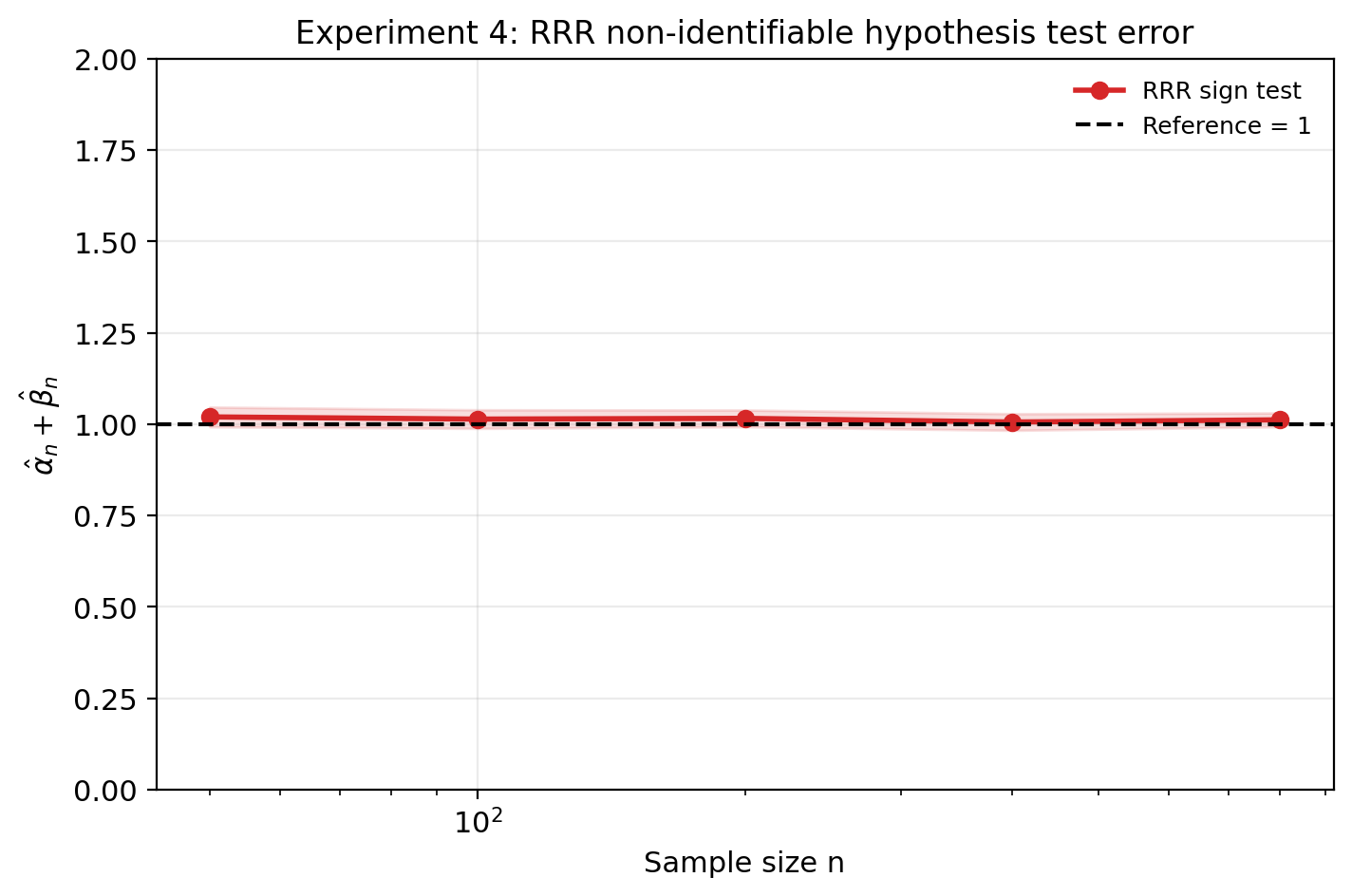}
\caption{
Empirical Type~I and Type~II errors for the
non-identifiable singular-vector sign hypothesis
in reduced-rank regression.
The sum of errors remains close to one
across increasing sample sizes,
illustrating the overlap-based impossibility result.
}
\label{fig:rrr-nonid}
\end{figure}

\section{Discussion and Conclusion}

Singular statistical models are often regarded as inherently problematic for inference. The results of this paper show that, for hypothesis testing,
the difficulty is not singularity itself,
but the formulation of the hypothesis. When a hypothesis depends on non-identifiable parameter quantities, the induced subsets of the model overlap in distribution space. In such cases, no uniformly consistent test exists, independently of sample size or estimation method. This obstruction is structural and cannot be resolved by improved algorithms or alternative asymptotic arguments.

In contrast, hypotheses formulated over identifiable observable regimes reduce entirely to classical testing theory. When the corresponding subsets of the model are separated in Hellinger distance, uniformly consistent tests exist and posterior contraction follows from standard arguments. Thus singularity does not preclude classical hypothesis testing;
what matters is whether the quantity of interest is identifiable at the level of induced distributions.

Near singular boundaries, separation may collapse locally, leading to scale-dependent detectability.
In such regimes, the feasibility of testing depends jointly on sample size and distance to the singular stratum. This behavior parallels classical local alternatives, but arises naturally from the geometry of singular models.

The numerical experiments illustrate the full structural classification. In both Gaussian mixture models and reduced-rank regression, non-identifiable parameter hypotheses remain untestable, while identifiable observable hypotheses exhibit classical
size control and increasing power.

The perspective developed here suggests a guiding principle for hypothesis formulation in singular models: testing should be posed in terms of identifiable observables, that is, quantities that factor through the induced distribution. Future work may explore sharper rate characterizations near singular strata and extensions to broader classes of models, but the central structural distinction is already clear.

\section*{Acknowledgments}
AI assistance statement: The author used a large language model to assist with editing, drafting, and organizational refinement of the manuscript. All mathematical arguments, proofs, experimental implementations, and interpretations are the author’s own work and were independently verified.

\bibliographystyle{plain}
\bibliography{refs}

\end{document}